\documentclass[12pt]{amsart}
\usepackage{amssymb,amsmath,amsfonts,latexsym,mathrsfs}
\usepackage{bm}

\usepackage{array, graphics, color}
\numberwithin{equation}{section}

\newtheorem{propn}{Proposition}[section]
\newtheorem{thm}[propn]{Theorem}
\newtheorem{lemma}[propn]{Lemma}

\newtheorem{cor}[propn]{Corollary}

\newtheorem*{thm*}{Theorem}
\theoremstyle{definition}
\newtheorem{defn}[propn]{Definition}
\newtheorem*{defn*}{Definition}

\newtheorem*{example}{Example}

\newtheorem{rem}{Remark}[section]

\newtheorem*{question}{Question}

\newcommand{\Nat}{\mathbb{N}}

 \newcommand{\D}{\mathbb{D}}
 
  \newcommand{\T}{\mathbb{T}}

 \DeclareMathOperator{\ran}{ran}

\newcommand{\clb}{\mathcal{B}}

\newcommand{\cle}{\mathcal{E}}
\newcommand{\clf}{\mathcal{F}}

\newcommand{\clh}{\mathcal{H}}
\newcommand{\clk}{\mathcal{K}}

\newcommand{\clm}{\mathcal{M}}

\newcommand{\cls}{\mathcal{S}}

\newcommand{\raro}{\rightarrow}

\begin{document}
\title[Unitary part of Toeplitz operators]{Unitary parts of Toeplitz operators with operator-valued symbols}


\author[Narayanan]{E. K. Narayanan}
\address{Department of Mathematics, Indian Institute of Science, Bangalore, 560012, India}
\email{naru@iisc.ac.in,
ek.narayanan@gmail.com}

\author[Sarkar]{Srijan Sarkar}
\address{Department of Mathematics, Indian Institute of Science, Bangalore, 560012, India}
\email{srijans@iisc.ac.in,
srijansarkar@gmail.com}

\begin{abstract}
Motivated by the canonical decomposition of contractions on Hilbert spaces, we investigate when contractive Toeplitz operators on vector-valued Hardy spaces on the unit disc admit a non-zero reducing subspace on which its restriction is a unitary. We show that for a Hilbert space $\mathcal{E}$ and operator-valued symbol $\Phi \in L_{\mathcal{B}(\mathcal{E})}^{\infty}(\mathbb{T})$, the Toeplitz operator $T_{\Phi}$ on $H_{\mathcal{E}}^2(\D)$ has such a unitary subspace if and only if there exists a Hilbert space $\mathcal{F}$, an inner function $\Theta(z) \in  H_{\mathcal{B}(\mathcal{F}, \mathcal{E})}^{\infty}(\D)$, and a unitary $U:\mathcal{F} \rightarrow \mathcal{F}$ such that
\[
\Phi(e^{it}) \Theta(e^{it}) = \Theta(e^{it}) U \quad \text{and} \quad  \Phi(e^{it})^* \Theta(e^{it}) = \Theta(e^{it}) U^* \quad (\text{ a.e. on }\mathbb{T}).
\]
This result can be seen as a generalization of the corresponding result for Toeplitz operators on $H^2(\mathbb{D})$ by Goor in \cite{Goor}.  We provide finer characterizations for analytic Toeplitz operators by finding the correspondence between the unitary parts of $T_{\Phi}$ on $H_{\mathcal{E}}^2(\mathbb{D})$ and $\Phi(0)$ on $\mathcal{E}$. 
\end{abstract}

\maketitle
\section{Introduction}
This article is concerned with finding the unitary parts of contractive Toeplitz operators with operator-valued symbols on vector-valued Hardy spaces.  For any Hilbert space $\cle$, the $\cle$-valued Hardy space on the unit disc $\D$ in the complex plane is defined by
\[
H_{\cle}^2(\D) := \{f = \sum_{k=0}^{\infty} a_n z^n \in \mathcal{O}(\D,\cle): a_n \in \cle, \|f\|:= \sum_{n=0}^{\infty} \|a_n\|_{\cle}^2 < \infty\}.
\]
The space of operator-valued bounded analytic functions on $\D$ is denoted by $H_{\clb(\cle)}^{\infty}(\D)$, where $\clb(\cle)$ denotes the class of all bounded operators on $\cle$. Let us also recall that $T \in \clb(\clh)$ is said to be a \textit{contraction} if $\|T\| \leq 1$.
Toeplitz operators are defined in the following manner.
\begin{defn}
A bounded linear operator $T$ on $H_{\cle}^2(\D)$ is said to be Toeplitz if there exists an operator-valued function $\Phi \in L_{\clb(\cle)}^{\infty}(\mathbb{T})$ such that $T = P_{H_{\cle}^2(\D)} L_{\Phi}|_{H_{\cle}^2(\D)},$ where $L_{\Phi}$ is the Laurent operator (multiplication operator) associated to $\Phi.$ In this case, $T$ is denoted by $T_{\Phi}$, and the function $\Phi$ is called the symbol of the operator $T_{\Phi}.$
\end{defn}
A starting point of several studies on operators involves finding an orthogonal decomposition of the Hilbert space into unitary and non-unitary parts. Following the celebrated monograph \cite{NF}, we define a completely non-unitary operator as the following:
\begin{defn}
$T \in \clb(\clh)$ is said to be completely non-unitary (c.n.u.) if there does not exist a $T$-reducing subspace $\clm \subseteq \clh$ such that $T|_{\clm}$ is a unitary.
\end{defn}
Contractions on Hilbert spaces admit a canonical orthogonal decomposition into unitary and c.n.u. parts as can be seen in the following result \cite{NF}.

\begin{thm}\label{contdeco}
Let $T$ be a contraction on $\clh$, then $ \clh = \clh_U \oplus \clh_{\text{c.n.u.}}$, where $\clh_U,\clh_{\text{c.n.u.}} \subseteq \clh$ are $T$-reducing closed subspaces, and therefore
\[
T = T|_{\clh_U} + T|_{\clh_{\text{c.n.u.}}}.
\]
Moreover, $T|_{\clh_U}$ is unitary and $T|_{\clh_{\text{c.n.u.}}}$ is  completely non-unitary.
\end{thm}

From its inception, c.n.u. contractions have played a major role in classifying contractions. The most striking result is the explicit functional model for c.n.u. contractions established by Nagy and Foias (\cite{NF}). We refer the reader to the articles \cite{Bercovici, BT, Das, Levan, MT, Wu} for several interesting aspects of c.n.u. contractions on Hilbert spaces.  Thus, it is a central question in operator theory to decide: \textit{when does a contraction become completely non-unitary?} It is evident from Theorem \ref{contdeco}, that the most natural approach to finding conditions for c.n.u. contractions depend on the understanding of the unitary subspace.  Motivated by this, let us put the following definition. 

\begin{defn*}
$T \in \clb(\clh)$ is said to have a \textit{unitary part} if there exists a non-zero $T$-reducing subspace $\clm_T \subseteq \clh$ such that $T|_{\clm_T}$ is unitary.
\end{defn*}

From the advent of the theory of Toeplitz operators following the work of Brown and Halmos in \cite{BH}, it is an intriguing and challenging question to find the correspondence between operator theoretic properties of $T_{\Phi}$ and function-theoretic properties of $\Phi$.  The understanding of this interplay is far from complete for the case of Toeplitz operators with operator-valued symbols as many natural questions remain unanswered. This motivates us to raise the following question.
\[
\textit{ Which Toeplitz operators on Hardy spaces on the unit disc admit a unitary part?}
\]
Goor answered this question (\cite{Goor}) for Toeplitz operators on $H^2(\D)$ by proving the following surprising result: \textit{any non-constant contractive $T_{\phi}$ is always completely non-unitary.}  Goor's result can be stated in the following equivalent manner.
\begin{thm*}[Goor]
Any Toeplitz operator $T_{\phi}$ on $H^2(\D)$ with a non-constant symbol $\phi$ and $\|T_{\phi}\|\leq 1$, does not admit a unitary part.
\end{thm*}
A natural question is to seek analogous results for Toeplitz operators with operator-valued symbols on vector-valued Hardy spaces.  A fundamental challenge that lies for operator-valued symbols is that, unlike the scalar-valued cases, we no longer have the commutativity of the symbols.  This is a reason why we need new approaches to tackle the general problem that often brings out a better understanding of the connection between a Toeplitz operator and its symbol. In this article, we have achieved the following result which gives a complete characterization for Toeplitz operators admitting a unitary part in $H_{\cle}^2(\D)$.
\begin{thm}
Let $T_{\Phi}$ be a contractive Toeplitz operator on $H_{\cle}^2(\D)$.  Then the following are equivalent
\begin{enumerate}
\item[(i)] $T_{\Phi}$ has a unitary part, say $\clm_{T_{\Phi}}$,
\item[(ii)] there exists a Hilbert space $\clf$, an inner function $\Theta(z) \in H_{\clb(\clf,\cle)}^{\infty}(\D)$ and a unitary $U: \clf \raro \clf$ such that
\[
\Phi(e^{it}) \Theta(e^{it}) = \Theta(e^{it}) U \quad \text{and} \quad  \Phi(e^{it})^* \Theta(e^{it}) = \Theta(e^{it}) U^* \quad (\text{ a.e. on }\T).
\]
In this case $T_{\Phi}|_{\clm_{T_{\Phi}}}$ is unitarily equivalent to a constant unitary.
\end{enumerate} 
\end{thm}
As expected, this result enables us to identify sufficient conditions for Toeplitz operators to become completely non-unitary (see Theorem \ref{cnu-toep}).

We point out that an extension of Goor's result to the case of vectorial Toeplitz operators was claimed by Butz \cite{Butz}. Unfortunately, there is a gap in the proof of the main theorem in \cite{Butz} which we point out in Section \ref{Butz}.
The author's result when applied to $T_{\Phi}$ on $H_{\cle}^2(\D)$ implies that the unitary part of $T_{\Phi}$ must be of the form $H_{\cle_0}^2(\D)$ for some closed subspace $\cle_0 \subseteq \cle$. We show in Theorem \ref{Butzcorrect} that this happens only when $T_{\Phi}$ restricted to its unitary part is equal to a constant unitary.  More precisely, we establish the following result.
\begin{thm}
Let $T_{\Phi}$ be a contractive Toeplitz operator on $H_{\cle}^2(\D)$ with a unitary part $\clm_{T_{\Phi}}$. Then the following are equivalent:
\begin{enumerate}
\item[(i)] $\clm_{T_{\Phi}} = H_{\cle_0}^2(\D)$ for some closed subspace $\cle_0 \subseteq \cle$,
\item[(ii)] there exists a closed subspace $\cle_0 \subseteq \cle$ such that $\clm_{T_{\Phi}} \subseteq H_{\cle_0}^2(\D)$, and $T_{\Phi}|_{\clm_{T_{\Phi}}} = T_W|_{H_{\cle_0}^2(\D)}$, for some unitary $W: \cle_0 \raro \cle_0$,
\item[(iii)] there exists a closed subspace $\cle_0 \subseteq \cle$ such that $\clm_{T_{\Phi}} \subseteq H_{\cle_0}^2(\D)$, and $\cle_0$ is $\Phi(e^{it})$-reducing and $\Phi(e^{it})|_{\cle_0}$ is a constant unitary a.e. on $\T$.
\end{enumerate}
\end{thm}
Note that $T_W$ is the constant Toeplitz operator  $W$ acting on the full Hardy space. Section \ref{Butz} is dedicated to proving the above result and also highlights the gap in the proof of the main theorem in \cite{Butz}.

Analytic Toeplitz operators, that is, Toeplitz operators with operator-valued bounded analytic symbols play a central role in connecting several topics with Operator theory.  There are two main reasons for this. The operator-theoretic advantage is that in this case, the Toeplitz operator becomes simply the multiplication operator.  For this reason, we choose to use the following notational convention throughout this article: 
\[
\text{If } \Phi(z) \in H_{\clb(\cle)}^{\infty}(\D) \text{ then the corrsponding Toeplitz operator is denoted by } M_{\Phi}.
\]
The other advantage lies in the fact that bounded analytic functions admit a realization formula. More precisely, any $\Phi(z) \in H_{\clb(\cle)}^{\infty}(\D)$ admits a transfer-function realization corresponding to a unitary
\[
W:= \begin{bmatrix}
A &B\\
C &D
\end{bmatrix}:\cle \oplus \clk \raro \cle \oplus \clk,
\]
that is, $\Phi(z) := \tau_W(z) = A +zB(I_{\cle} - zD)^{-1}C$. This way of representing bounded analytic functions has a profound impact on several topics in operator theory and function theory.  In particular, analytic Toeplitz operators with no unitary part play a vital role in the characterization of distinguished varieties in the bidisc by Agler and McCarthy in \cite{AM}. Their result shows that a distinguished variety in the bidisc always admits the following description 
\[
\mathcal{V} = \{(z,w) \in \D^2: \det (\Phi(z) - wI)=0\},
\]
where $\Phi(z)$ is a matrix-valued rational inner function in $H_{\clb(\mathbb{C}^n)}^{\infty}(\D)$.  The boundary behaviour of the distinguished variety is closely related to the fact that $M_{\Phi}$ on $H_{\mathbb{C}^n}^2(\D)$ does not admit a unitary part. This connection has been further explored in some recent works \cite{DS, DSau}.

Thus, it is evident that characterizing Toeplitz operators on $H_{\cle}^2(\D)$ with unitary parts will have far-reaching consequences.  Through the realization formula of $\Phi(z)$ it is easy to observe that the existence of the unitary part of $\Phi(0)$ will always imply that $M_{\Phi}$ will have a unitary part. Several examples in the literature show that the converse does not hold. For instance, see \cite[Remark 3.2]{DSau}. So, the main challenge lies in identifying the conditions which together with the existence of the unitary part of $M_{\Phi}$ will imply that $\Phi(0)$ will have a unitary part. This we have achieved in this article by showing the following result.

\begin{thm*}
Let $\Phi(z) \in H_{\clb(\cle)}^{\infty}(\D)$ with $\|\Phi\|_{\infty} \leq 1$.  If $M_{\Phi}$ has an unitary part $\clm_{M_{\Phi}}$, then the condition $\|T_{\Phi(0)} f \| = \|f\|$ for all $f \in \clm_{M_{\Phi}}$ implies that $\Phi(0)$ has a  unitary part.
\end{thm*}
This result shows that for investigating the unitary parts of $M_{\Phi}$, we need to consider $\Phi(0)$ as an operator on the full Hardy space. In particular,  for the existence of the unitary part of $\Phi(0)$ in $\cle$, we need to check whether $\Phi(0)$ is an isometry on the unitary part of $M_{\Phi}$.

Let us now discuss the plan of this paper. In section \ref{preliminaries}, we briefly give some background details and establish a few results that will be used in later sections. Following this, in section \ref{unitary}, we establish our main results involving the unitary parts of Toeplitz operators. In this section, we also discuss Butz's characterization and related results. In section \ref{analytic-toeplitz}, we deduce refined versions for analytic Toeplitz operators.  In the last section \ref{conclusion}, we discuss the conditions that are sufficient for Toeplitz operators to become completely non-unitary. We end with some relevant questions which are interesting in their own right.


\section{Preliminaries}\label{preliminaries}

This section will establish the background details that will be used throughout this article. Let us begin by looking at vector-valued Hilbert spaces. More precisely, let $\cle$ be a Hilbert space, then the space of square-integrable $\cle$-valued functions is denoted by $L_{\cle}^2(\mathbb{T})$, that is, 
\[
L_{\cle}^2(\mathbb{T}):= \{f \text{ is a } \cle-\text{valued measurable function on } \T \text{ and } \|f\|_2:= \big( \int_{\mathbb{T}} {\|f(z)\|_{\cle}^2 } \hspace{2mm} \mathrm{d} \mu \big)^{\frac{1}{2}} < \infty \},
\]
where $\mu$ is the normalized Lebesgue measure on $\mathbb{T}$. The space of all essentially bounded $\clb(\cle)$-valued functions on $\T$ is denoted by $L_{\clb(\cle)}^{\infty}(\T)$. Given $\Phi \in L_{\clb(\cle)}^{\infty}(\T)$ the Laurent operator $L_{\Phi}$ on $L_{\cle}^2(\T)$ is defined by 
\[
L_{\Phi} f(z) = \Phi(z) f(z) \quad (z \in \T).
\]
Brown and Halmos observed in \cite{BH} that an operator $T$ on $H^2(\D)$ is a Toeplitz operator if and only if $M_z^* T M_z = T$, where $M_z$ is the shift operator on $H^2(\D)$. The same characterization holds for Toeplitz operators on $H_{\cle}^2(\D)$, see for instance \cite{Page}. This algebraic characterization immediately gives us the following (probably well-known) results.
\begin{thm}\label{isometry}
Let $T_{\Phi}$ be a Toeplitz operator on $H_{\cle}^2(\D)$. Then $T_{\Phi}$ is an isometry if and only if $\Phi(z)$ is an inner function in $H_{\clb(\cle)}^{\infty}(\D)$. 
\end{thm}
\begin{proof}
Since $T_{\Phi}$ is an isometry, we get
\begin{align*}
I_{H_{\cle}^2(\D)}  = M_{z}^*T_{\Phi}^* T_{\Phi}M_z = M_{z}^*T_{\Phi}^* M_z M_z^* T_{\Phi}M_z + M_{z}^*T_{\Phi}^* P_{\cle} T_{\Phi}M_z &= T_{\Phi}^* T_{\Phi} + M_{z}^*T_{\Phi}^* P_{\cle} T_{\Phi}M_z \\
&= I_{H_{\cle}^2(\D)} + M_{z}^*T_{\Phi}^* P_{\cle} T_{\Phi}M_z,
\end{align*}
where $P_{\cle}$ is the projection to $\cle$ viewed as constant functions in $H_{\cle}^2(\D)$.
Therefore,we get $M_{z}^*T_{\Phi}^* P_{\cle} T_{\Phi}M_z=0$, which is equivalent to $P_{\cle} T_{\Phi}M_z=0$. In other words,
\[
0 = P_{\cle} T_{\Phi}M_z = \big(I_{H_{\cle}^2(\D)} -M_z M_z^* \big)T_{\Phi}M_z = T_{\Phi}M_z - M_z T_{\Phi}.
\]
This implies that $\Phi(z)$ must be in $H_{\clb(\cle)}^{\infty}(\D)$ and furthermore, $T_{\Phi}$ is an isometry forces $\Phi(z)$ to be an inner function.
\end{proof}
Recall that $\Phi \in H_{\clb(\cle)}^{\infty}(\D)$ is defined to be an inner function if $\Phi(e^{it})$ is an isometry a.e. on $\T$. Applying the above result for unitary Toeplitz operators gives the following.
\begin{cor}\label{unitoep}
Let $T_{\Phi}$ be a Toeplitz operator on $H_{\cle}^2(\D)$. Then $T_{\Phi}$ is a unitary if and only if $\Phi(z)$ is a constant unitary on $\cle$.
\end{cor}
%
Before ending this section, let us note a result that will be used in the later sections.
\begin{lemma}\label{isom}
Let $A$ be a bounded operator on $\clh$ and $V \in \clb(\clh)$ be an isometry. Then $V \clh$ is $A$-reducing if and only if 
\[
A VV^* = VV^* A.
\]
\end{lemma}
\begin{proof}
First, let us note that $A VV^* = VV^* A$ implies that $V\clh$ is $A$-reducing closed subspace of $\clh$. Now suppose that $V\clh$ is $A$-reducing then it immediately implies that 
\[
VV^* A V = AV,
\]
which further implies that $VV^* A VV^* = AVV^*$ and therefore,
\[
VV^* A (I - P_{\ker V^*}) = AVV^*.
\]
Now, $V\clh$ is $A$-reducing implies that $\ker V^*$ is $A$-reducing and therefore, $VV^* A P_{\ker V^*} = 0$. Hence,
\[
VV^* A (I - P_{\ker V^*}) = VV^* A.
\]
This completes the proof.
\end{proof}

\section{Unitary parts of Toeplitz operators}\label{unitary}
In this section, we aim to characterize the unitary parts of Toeplitz operators on vector-valued Hardy spaces.  Let us begin by first analyzing the unitary parts of the Laurent operators $L_{\Phi}$ on $L_{\cle}^2(\T)$. 

\begin{thm}
A contractive Laurent operator $L_{\Phi}$ on $L_{\cle}^2(\T)$ has a unitary part if and only if there exists a measurable subset $N \subseteq \mathbb{T}$ of positive measure such that $\Phi(z)$ is a unitary a.e. on $N$.
\end{thm}
\begin{proof}
Note that the unitary subspace of the contractive Laurent operator $L_{\Phi}$ is 
\[
\clm_{L_{\Phi}} = \{h \in L_{\cle}^2(\T): \|L_{\Phi}^{n}h\| = \|h\| = \|L_{\Phi}^{*n}h\|, \forall n \in \Nat \},
\]
equivalently, 
\[
\clm_{L_{\Phi}} = \{h \in L_{\cle}^2(\T): L_{\Phi}^{n} L_{\Phi}^{*n}h = h = L_{\Phi}^{*n} L_{\Phi}^{n} h, \forall n \in \Nat \},
\]
Since $L_{\Phi}$ commutes with the bilateral shift $B$ on $L_{\cle}^2(\T)$ we get that $\clm_{L_{\Phi}}$ is a $B$-reducing closed subspace of $L_{\cle}^2(\mathbb{T})$. From \cite[Theorem 3.6, Page 39]{RR} it follows that there exists a measurable set $N \subseteq \T$ such that
\[
\clm_{L_{\Phi}} = \{f \in L_{\cle}^2(\T): f = 0 \text{ a.e. on N} \} = \ran L_{\chi_{N^{\complement}}},
\]
where $\chi_{N^{c}}$ is the characteristic function of $N^c:= \T \setminus N$. Thus, for any $f \in L_{\cle}^2(\T)$ we have
\[
L_{\Phi}^{*n} L_{\Phi}^{n}  L_{\chi_{N^c}} f= L_{\chi_{N^c}}f; \quad L_{\Phi}^{n}  L_{\Phi}^{*n} L_{\chi_{N^c}} f= L_{\chi_{N^c}}f \quad (\forall n \in \Nat).
\]
In other words,  for all $n \in \Nat$,
\[
\Phi(e^{it})^{n} \Phi(e^{it})^{*n} \chi_{N^c}(e^{it}) = \chi_{N^c}(e^{it}) ; \quad \Phi(e^{it})^{*n} \Phi(e^{it})^n \chi_{N^c}(e^{it}) = \chi_{N^c}(e^{it}) \quad (\text{ a.e. on } \mathbb{T}).
\]
The converse direction is straightforward. This completes the proof.
\end{proof}
We will now prove the main result of this article.
\begin{thm}\label{mainthm}
Let $T_{\Phi}$ be a contractive Toeplitz operator on $H_{\cle}^2(\D)$.  Then the following are equivalent
\begin{enumerate}
\item[(i)] $T_{\Phi}$ has a unitary part say $\clm_{T_{\Phi}}$,
\item[(ii)] there exists a Hilbert space $\clf$ and an inner function $\Theta(z) \in H_{\clb(\clf,\cle)}^{\infty}(\D)$ and a unitary $U: \clf \raro \clf$ such that
\begin{equation}\label{maincondn}
\Phi(e^{it}) \Theta(e^{it}) = \Theta(e^{it}) U \quad \text{and} \quad  \Phi(e^{it})^* \Theta(e^{it}) = \Theta(e^{it}) U^* \quad (\text{ a.e. on }\T).
\end{equation}
In this case $T_{\Phi}|_{\clm_{T_{\Phi}}}$ is unitarily equivalent to a constant unitary.
\end{enumerate} 
\end{thm}
\begin{proof}
Since $T_{\Phi} = P_{H_{\cle}^2(\D)} L_{\Phi}|_{H_{\cle}^2(\D)}$, the assumption $\|T_{\Phi}\| \leq 1$, implies that $\|L_{\Phi}\| = \|T_{\Phi}\| \leq 1$.  Let us now prove $(i) \implies (ii)$.  From the assumption in $(i)$, the following unitary part of $T_{\Phi}$ is non-zero
\[
\clm_{T_{\Phi}} = \{h \in H_{\cle}^2(\D): \|T_{\Phi}^{n}h\| = \|h\| = \|T_{\Phi}^{*n}h\|, \forall n \in \Nat \}
\]
Now, $h \in \clm_{T_{\Phi}}$ implies that
\[
\|h\| = \|T_{\Phi}h\| =\|P_{H_{\cle}^2(\D)} L_{\Phi} h\| \leq \| L_{\Phi} h\| \leq \|h\|,
\]
and therefore, $ L_{\Phi} h = T_{\Phi}h$. This implies that $L_{\Phi}h \in H_{\cle}^2(\D)$ and $\|L_{\Phi}h\| = \|h\|$. In a similar manner, we can show that $L_{\Phi}^* h \in H_{\cle}^2(\D)$ and $\|L_{\Phi}^*h \| = \|h\|$.  Now, we will prove by induction that
\[
\clm_{T_{\Phi}} = \{h \in H_{\cle}^2(\D): L_{\Phi}^{n}h, L_{\Phi}^{*n}h \in H_{\cle}^2(\D) \text{ and } \|L_{\Phi}^{n}h\| = \|h\| = \|L_{\Phi}^{*n}h\|, \forall n \in \Nat \}
\]
Suppose $k \in \Nat$ such that for all $n \in \Nat$ and $1 \leq n \leq k$, the following,
\[
L_{\Phi}^{n}h, L_{\Phi}^{*n}h \in H_{\cle}^2(\D) \text{ and } \|L_{\Phi}^{n}h\| = \|h\| = \|L_{\Phi}^{*n}h\|,
\]
holds for all $h \in \clm_{T_{\Phi}}$. Then 
\[
\|h\| = \|T_{\Phi}^{k+1}h\| =\|(P_{H_{\cle}^2(\D)} L_{\Phi} |_{H_{\cle}^2(\D)})^{k+1} h\| \leq \| L_{\Phi}^{k+1} h\| \leq \|h\|,
\]
shows that $\| L_{\Phi}^{k+1} h\| = \|T_{\Phi}^{k+1}h\| = \|h\|$ and therefore, $L_{\Phi}^{k+1} h \in H_{\cle}^2(\D)$. Similarly, we can show that $\| L_{\Phi}^{*k+1} h\| = \|T_{\Phi}^{*k+1}h\| = \|h\|$ and therefore, $L_{\Phi}^{*k+1} h \in H_{\cle}^2(\D)$. This proves our claim about $\clm_{T_{\Phi}}$. Furthermore, $\|L_{\Phi}^*\| = \|L_{\Phi}\| = \|T_{\Phi}\| \leq 1$, implies that for any $h \in \clm_{T_{\Phi}}$
\[
\|( I_{H_{\cle}^2(\D)} - L_{\Phi}^n L_{\Phi}^{*n})^{\frac{1}{2}}h\|^2 = \langle I_{H_{\cle}^2(\D)} - L_{\Phi}^n L_{\Phi}^{*n}h, h \rangle  = \|h\|^2 - \|L_{\Phi}^{*n} h\|^2 = 0,
\]
and similarly,
\[
\|( I_{H_{\cle}^2(\D)} - L_{\Phi}^{*n} L_{\Phi}^{n})^{\frac{1}{2}}h\|^2 = \langle I_{H_{\cle}^2(\D)} - L_{\Phi}^{*n} L_{\Phi}^{n}h, h \rangle  = \|h\|^2 - \|L_{\Phi}^n h\|^2 = 0.
\]
Thus, we get that 
\begin{equation}\label{structure}
\clm_{T_{\Phi}} = \{h \in H_{\cle}^2(\D): L_{\Phi}^{n}h, L_{\Phi}^{*n}h \in H_{\cle}^2(\D) \text{ and } L_{\Phi}^{n}L_{\Phi}^{*n}h= h = L_{\Phi}^{*n}L_{\Phi}^{n}h, \forall n \in \Nat \}
\end{equation}
From the above structure, we can immediately identify the following properties.
\begin{equation}\label{cnu1}
L_{\Phi} \clm_{T_{\Phi}} = \clm_{T_{\Phi}}; \quad L_{\Phi}^* \clm_{T_{\Phi}} = \clm_{T_{\Phi}}.
\end{equation}
Now, for any $h \in \clm_{T_{\Phi}}$ and $n \in \Nat$,
\[
L_{\Phi}^n M_z h = L_{\Phi}^n B h = B L_{\Phi}^n h = M_z L_{\Phi}^n h,
\]
holds, where $B$ is the bilateral shift on $L_{\cle}^2(\T)$, and thus $M_z = B|_{H_{\cle}^2(\D)}$. In a similar manner, $L_{\Phi}^{*n} M_z h = M_z L_{\Phi}^{*n} h$. Thus, for any $h \in \clm_{T_{\Phi}}$ we get,
\[
\| L_{\Phi}^n M_z h \| = \|L_{\Phi}^n h\| = \|h\| =  \|L_{\Phi}^{*n} h\| = \| L_{\Phi}^{*n} M_z h \|.
\]
The above observations prove that $\clm_{T_{\Phi}}$ is a $M_z$-invariant closed subspace of $H_{\cle}^2(\D)$. Thus, we can use the Beurling-Lax-Halmos theorem to show that there exists a Hilbert space $\clf$ and an inner function $\Theta(z) \in H_{\clb(\clf, \cle)}^{\infty}(\D)$ such that
\[
\clm_{T_{\Phi}} = \Theta(z) H_{\clf}^2(\D).
\]
From this description of $\clm_{T_{\Phi}}$ and using condition (\ref{structure}) we get 
\begin{equation}\label{prop1}
L_{\Phi}^* L_{\Phi} M_{\Theta} = M_{\Theta} = L_{\Phi} L_{\Phi}^* M_{\Theta},~\text{on}~H_{\clf}^2(\D).
\end{equation}
Now, 
\[
L_{\Phi} M_{\Theta} M_z = L_{\Phi} M_z M_{\Theta} = L_{\Phi} B M_{\Theta} = B L_{\Phi} M_{\Theta}  = M_z L_{\Phi} M_{\Theta},
\]
where the last equality follows from the fact $ L_{\Phi} M_{\Theta} H_{\clf}^2(\D) \subseteq H_{\cle}^2(\D)$. This implies that 
\begin{equation}\label{condn1}
L_{\Phi} M_{\Theta}  = M_{\Psi_1},
\end{equation}
for some $\Psi_1 \in H_{\clb(\clf,\cle)}^{\infty}(\D)$ and thus, from condition (\ref{prop1}) we also get that,
\begin{equation}\label{condn4}
 M_{\Theta}  = L_{\Phi}^* M_{\Psi_1},
\end{equation}
Moreover, for any $f \in H_{\clf}^2(\D)$,
\[
\|M_{\Psi_1} f\| = \|L_{\Phi} M_{\Theta} f\| = \|f\|,
\]
shows that $\Psi_1(z)$ must be an inner function. Similarly, we can show that there exists an inner function $\Psi_2(z) \in H_{\clb(\clf,\cle)}^{\infty}(\D)$ such that
\begin{equation}\label{condn2}
L_{\Phi}^* M_{\Theta} = M_{\Psi_2}.
\end{equation}
and thus, using condition (\ref{prop1}), we get
\begin{equation}\label{ancondn1}
 M_{\Theta} = L_{\Phi} M_{\Psi_2}.
\end{equation}
Now using conditions (\ref{cnu1}) and (\ref{condn2}), we get that
\[
M_{\Psi_2} H_{\clf}^2(\D) = L_{\Phi}^* M_{\Theta} H_{\clf}^2(\D) = L_{\Phi}^*\clm_{T_{\Phi}} = \clm_{T_{\Phi}} = M_{\Theta}H_{\clf}^2(\D).
\]
By the uniqueness of the Beurling-Lax-Halmos theorem \cite[Theorem 2.1, Page 239]{FoFr}, there must exist a unitary (say) $U:\clf \raro \clf$ such that
\begin{equation}\label{condn3}
\Theta(z) = \Psi_2(z) U \quad (z \in \D).
\end{equation}
Moreover, from conditions (\ref{condn1}), (\ref{ancondn1}), and (\ref{condn3}) we have
\[
M_{\Psi_1} = L_{\Phi} M_{\Theta} = L_{\Phi} M_{\Psi_2} U = M_{\Theta}U.
\]
In other words, 
\begin{equation}\label{condn5}
\Theta(z) = \Psi_1(z) U^*  \quad (z \in \D).
\end{equation}
Now, from conditions (\ref{ancondn1}) and (\ref{condn3}) we get,
\begin{equation}\label{imp1}
\Phi(e^{it}) \Theta(e^{it}) =  \Phi(e^{it}) \Psi_2(e^{it}) U  = \Theta(e^{it}) U \quad (\text{ a.e. on }\T).
\end{equation}
and furthermore, from conditions (\ref{condn4}) and (\ref{condn5}) we get
\begin{equation}\label{imp2}
\Phi(e^{it})^* \Theta(e^{it}) =\Phi(e^{it})^* \Psi_1(e^{it}) U^* = \Theta(e^{it}) U^* \quad (\text{ a.e. on }\T).
\end{equation}

From the above conditions, we can observe that 
\[
M_{\Theta}^* T_{\Phi} M_{\Theta} = U; \quad M_{\Theta}^* T_{\Phi}^* M_{\Theta} = U^*,
\]
which shows that $T_{\Phi}|_{\clm_{T_{\Phi}}}$ is unitarily equivalent to a constant unitary. 

Conversely, if there exists a Hilbert space $\clf$, for which there exists an inner function $\Theta(z) \in H_{\clb(\clf, \cle)}^{\infty}(\D)$ and an unitary $U: \clf \raro \clf$ such that
\[
\Phi(e^{it}) \Theta(e^{it}) = \Theta(e^{it}) U \quad \text{and} \quad  \Phi(e^{it})^* \Theta(e^{it}) = \Theta(e^{it}) U^* \quad (\text{ a.e. on }\T).
\]
then it further implies that for all $n \in \Nat$, we have
\begin{equation}
\Phi(e^{it})^n \Theta(e^{it}) =  \Theta(e^{it}) U^{n}; \quad \Phi(e^{it})^{*n} \Theta(e^{it}) =  \Theta(e^{it}) U^{*n} \quad (\text{ a.e. on }\T).
\end{equation}
and therefore,
\begin{equation}
\Phi(e^{it})^{*n} \Phi(e^{it})^n \Theta(e^{it}) =  \Theta(e^{it}); \quad \Phi(e^{it})^n \Phi(e^{it})^{*n} \Theta(e^{it}) =  \Theta(e^{it}) \quad (\text{ a.e. on }\T).
\end{equation}
This shows that for all $n \in \Nat$, $L_{\Phi}^{n} M_{\Theta}$ and $L_{\Phi}^{*n} M_{\Theta}$ both belongs to $H_{\cle}^2(\D)$ and moreover,
\[
L_{\Phi}^{*n} L_{\Phi}^{n} M_{\Theta} = M_{\Theta} \quad \text{and} \quad  L_{\Phi}^n  L_{\Phi}^{*n} M_{\Theta} = M_{\Theta} \quad (\forall n \in \Nat),
\]
Thus clearly, $M_{\Theta} H_{\clf}^2(\D)\subseteq \clm_{T_{\Phi}}$ and therefore, $\clm_{T_{\Phi}} \neq \{0\}$, which implies that $T_{\Phi}$ has a non-zero unitary part. This completes the proof.
\end{proof}

The above result can be used to prove the following equivalent statement of Goor's characterization of completely non-unitary Toeplitz contractions in \cite{Goor}.
\begin{cor}
If $\phi \in L^{\infty}(\T)$, $\|\phi\| \leq 1$ and $\phi$ is not constant a.e. on $\T$, then $T_{\phi}$ does not have a unitary part.
\end{cor}
\begin{proof}
From the assumption, $T_{\phi}$ is a contractive Toeplitz operator on $H^2(\D)$. Now, if there exists a unitary part of $T_{\phi}$, then by the above result there must exist an inner function $\theta(z) \in H^{\infty}(\D)$ such that 
\[
\phi(e^{it}) \theta(e^{it}) =   \mu  \theta(e^{it}); \quad \phi(e^{it})^* \theta(e^{it}) =\bar{\mu}  \theta(e^{it}), \quad (\text{a.e. on } \T),
\]
 for some uni-modular constant $\mu$. This immediately implies that $\phi(e^{it}) = \bar{\mu}$ a.e. on $\T$, which is a contradiction to our assumption that $\phi$ is non-constant a.e. on $\T$. This completes the proof. 
\end{proof} 

\begin{rem}
Note that from Theorem \ref{mainthm}, if $T_{\Phi}$ has a unitary part, then there exists a Hilbert space $\clf$ and an inner function say $\Theta(z) \in H_{\clb(\clf,\cle)}^{\infty}(\D)$, such that
\[
\Theta(e^{it}) \clf \subseteq \{\eta \in \cle: \Phi(e^{it})^n \Phi^{*n}(e^{it}) \eta = \eta = \Phi(e^{it})^{*n} \Phi^{n}(e^{it}) \eta, \forall n \in \Nat\} \quad (\text{ a.e. on }\T),
\]
However, it is unclear when the above two sets are equal a.e. on $\T$.  For this reason, we cannot deduce that $T_{\Phi}$ has a non-zero unitary part if and only if $\Phi(e^{it})$ has a non-zero unitary part almost everywhere on $\T$. 
\end{rem}

\begin{rem}
The unitary part of a contraction $T$ on $\clh$ is the intersection of the co-isometry and isometry parts, that is,
\[
\clh_U(T) = \clh_i(T) \cap \clh_{i}(T^*),
\] 
where
\[
\clh_i(T) = \{h \in \clh: \|T^n h\| = \|h\|, \text{ for all $n$ } \in \Nat\}
\]
The method of proving Theorem \ref{mainthm} shows that if $T_{\Phi}$ has a isometry part, then it is also $M_z$-invariant and there exists a Hilbert space $\clf$, an inner function $\Theta(z) \in H_{\clb(\clf,\cle)}^{\infty}(\D)$, and a unitary $U:\clf \raro \clf$ such that
\[
\Phi(e^{it}) \Theta(e^{it}) = \Theta(e^{it}) U \quad (\text{ a.e. on } \mathbb{T}).
\]
Similarly, if $T_{\Phi}$ has a co-isometry part, then there exists a Hilbert space $\tilde{\clf}$, an inner function $\Psi(z)  \in H_{\clb(\tilde{\clf},\cle)}^{\infty}(\D)$, and a unitary $\tilde{U}^*:\tilde{\clf} \raro \tilde{\clf}$ such that 
\[
\Phi(e^{it})^* \Psi(e^{it}) = \Psi(e^{it}) \tilde{U}^* \quad (\text{ a.e. on } \mathbb{T}).
\]
Theorem \ref{mainthm} shows that if $T_{\Phi}$ has a unitary part then we can choose the same inner function and unitary in the above identities.
\end{rem}

Let us now identify the restriction imposed on $T_{\Phi}$ by the conditions (\ref{maincondn}) on the symbol $\Phi$.
\begin{propn}\label{mainprop}
Let $\Phi \in L_{\clb(\cle)}^{\infty}(\T)$ then the following statements are equivalent:
\begin{enumerate}
\item[(i)] There exists a Hilbert space $\clf$, an inner function $\Theta(z) \in H_{\clb(\clf,\cle)}^{\infty}(\D)$, and a unitary $U: \clf \raro \clf$ such that
\[
\Phi(e^{it}) \Theta(e^{it}) = \Theta(e^{it}) U \quad \text{and} \quad  \Phi(e^{it})^* \Theta(e^{it}) = \Theta(e^{it}) U^* \quad (\text{ a.e. on } \T).
\]
\item[(ii)] There exists a $M_z$-invariant closed subspace $\cls \subseteq H_{\cle}^2(\D)$ which is $T_{\Phi}$-reducing and $T_{\Phi}|_{ \cls}$ is a unitary.  In this case, $T_{\Phi}|_{\cls}$ is always unitarily equivalent to a constant unitary.
\end{enumerate}
\end{propn}
\begin{proof}
First, let us look at the proof of $(i) \implies (ii)$.  Note that, 
\begin{equation}
\Phi(e^{it}) \Theta(e^{it}) = \Theta(e^{it}) U \quad \text{and} \quad  \Phi(e^{it})^* \Theta(e^{it}) = \Theta(e^{it}) U^*,
\end{equation}
a.e. on $\T$ implies that 
\begin{equation}\label{toepcon1}
T_{\Phi} M_{\Theta} = M_{\Theta} U; \quad T_{\Phi}^* M_{\Theta} = M_{\Theta} U^*.
\end{equation}
This implies that $M_{\Theta} H_{\clf}^2(\D)$ is $T_{\Phi}$-reducing. Furthermore, 
\[
T_{\Phi}^n  T_{\Phi}^{*n} M_{\Theta} = M_{\Theta}; \quad T_{\Phi}^{*n}  T_{\Phi}^{n} M_{\Theta} = M_{\Theta},
\]
shows that $T_{\Phi}|_{M_{\Theta}H_{\clf}^2(\D)}$ is a unitary.  Now, from condition \ref{toepcon1}, we get
\[
M_{\Theta}^* T_{\Phi} M_{\Theta} = U; \quad M_{\Theta}^* T_{\Phi}^* M_{\Theta} = U^*.
\]
This implies that 
\[
T_{\Phi}|_{M_{\Theta} H_{\clf}^2(\D)} \cong U; \quad T_{\Phi^*}|_{M_{\Theta} H_{\clf}^2(\D)} \cong U^*.
\]
Thus, $T_{\Phi}|_{M_{\Theta} H_{\clf}^2(\D)}$ is unitarily equivalent to the constant unitary $U$.

For proving that $(ii) \implies (i)$, let us note that there exists a Hilbert space $\clf$ and an inner function $\Theta(z) \in H_{\clb(\clf,\cle)}^{\infty}(\D)$ such that $\cls = M_{\Theta} H_{\clf}^2(\D)$. Using Lemma \ref{isom} and the fact that $\cls$ is $T_{\Phi}$-reducing we get that 
\[
T_{\Phi} M_{\Theta}M_{\Theta}^* = M_{\Theta}M_{\Theta}^* T_{\Phi}.
\]
Now note that,
\[
T_{\Theta^* \Phi \Theta} = M_{\Theta}^* T_{\Phi} M_{\Theta},
\]
and therefore,
\begin{align*}
T_{\Theta^* \Phi \Theta}^* T_{\Theta^* \Phi \Theta} =  M_{\Theta}^* T_{\Phi}^* M_{\Theta}M_{\Theta}^* T_{\Phi} M_{\Theta} = M_{\Theta}^* T_{\Phi}^* T_{\Phi} M_{\Theta} = I_{H_{\clf}^2(\D)}.
\end{align*}
The last equality follows from the assumption that $T_{\Phi}|_{\cls}$ is unitary.  Thus, we get that $T_{\Theta^* \Phi \Theta}$ is an isometry, and hence, from Theorem \ref{isometry}, we get
\begin{equation}\label{inner1}
\Theta(z)^*  \Phi(z) \Theta(z) \quad \text{ is an inner function in } H_{\clb(\clf,\cle)}^{\infty}(\D).
\end{equation}
Similarly, using the assumption of reducibility, we get
\[
T_{\Phi}^* M_{\Theta}M_{\Theta}^* = M_{\Theta}M_{\Theta}^* T_{\Phi}^*,
\]
and therefore,
\begin{align*}
T_{\Theta^* \Phi^* \Theta}^* T_{\Theta^* \Phi^* \Theta} =  M_{\Theta}^* T_{\Phi} M_{\Theta} M_{\Theta}^* T_{\Phi}^* M_{\Theta}  = M_{\Theta}^* T_{\Phi} T_{\Phi}^* M_{\Theta} = I_{H_{\clf}^2(\D)}.
\end{align*}
This implies that $T_{\Theta^* \Phi^* \Theta}$ is an isometry, and by using Theorem \ref{isometry}, we get
\begin{equation}\label{inner2}
\Theta(z)^*  \Phi(z)^* \Theta(z) \quad \text{ is an inner function in } H_{\clb(\clf,\cle)}^{\infty}(\D).
\end{equation}
Thus, from conditions (\ref{inner1}) and (\ref{inner2}) we can deduce that 
\[
\Theta(z)^*  \Phi(z) \Theta(z) \quad \text{ is a constant unitary (say) } \quad U:\clf \raro \clf,
\]
which imply that
\[
M_{\Theta}^* T_{\Phi}^* M_{\Theta} = U^*; \quad M_{\Theta}^* T_{\Phi} M_{\Theta} = U.
\]
Applying by $M_{\Theta}$ on the left side of both identities we get
\[
T_{\Phi}^* M_{\Theta} = M_{\Theta} U^*; \quad  T_{\Phi} M_{\Theta} = M_{\Theta} U,
\]
and therefore, 
\[
\Phi(e^{it}) \Theta(e^{it}) = \Theta(e^{it}) U \quad \text{and} \quad  \Phi(e^{it})^* \Theta(e^{it}) = \Theta(e^{it}) U^*,
\]
almost everywhere on $\T$.  This completes the proof.
\end{proof}

\begin{rem}
In a sense, conditions (\ref{maincondn}) are associated with an operator equation related to Toeplitz operators. More precisely, if $T_{\Phi}$ is a contractive Toeplitz operator and $\clm_{\T_{\Phi}} = M_{\Theta}H_{\clf}^2(\D)$ is the unitary part, then the following equation for 
\[
\Phi \Theta = \Theta X; \quad \Phi^* \Theta = \Theta Y,
\]
always have a unique solution, i.e. $X= U$, and $Y = U^*$.
\end{rem}

\subsection{A note on  Butz's characterization}\label{Butz}
In this section, we show that Butz's result \cite{Butz}, for Toeplitz operators on vector-valued Hardy spaces holds only for some special cases. Note that, the author's result when applied to $T_{\Phi}$ on $H_{\cle}^2(\D)$ gives the following.
\begin{thm}[Butz] 
Let $T_{\Phi}$ be a Toeplitz operator on $H_{\cle}^2(\D)$. Then there exists an orthogonal  decomposition $\cle = \cle_0 \oplus \cle_1$ such that 
\[
\clm_{T_{\Phi}} = H_{\cle_0}^2(\D); \quad \clm_{T_{\Phi}}^{\perp} = H_{\cle_1}^2(\D), 
\]
and $x = \sum_{n=0}^{\infty} M_z^n f_n \in \clm_{T_{\Phi}}$ if and only if $f_n \in \cle_0$ for all $n \in \Nat \cup \{0\}$. Furthermore, if $\clm_{T_{\Phi}} \neq \{0\}$ (i.e. $ \cle_0 \neq \{0\}$), then there exists a unitary operator $R_0: \cle_0 \raro \cle_0$ such that for $x = \sum_{n=0}^{\infty} M_z^n f_n \in \clm_{T_{\Phi}}$, we have $T_{\Phi}x = \sum_{n=0}^{\infty} M_z^n R_0 f_n$.
\end{thm}
The following example shows that the above result is not always true.

\begin{example}
Consider $\cle = \cle_0 \oplus \cle_1$ a non-trivial orthogonal decomposition. Let $P = P_{\cle_0} \in \clb(\cle)$ and thus, $P^{\perp} = P_{\cle_1}$.  Now let us consider the symbol $\Phi(z):= zP+P^{\perp}$. Note that 
\[
I_{\cle} - \Phi(z)^* \Phi(z) = (1 - |z|^2)P = I_{\cle} - \Phi(z) \Phi(z)^*,
\]
shows that $\Phi(z)$ is unitary-valued on $\T$ and therefore, the corresponding Toeplitz operator $M_{\Phi}$ is an isometry. Since, $\Phi(0) = P^{\perp}$ is not a pure contraction on $\cle$, we know from \cite[Theorem 3.4]{Sarkar} that $M_{\Phi}$ is not completely non-unitary. In other words, the subspace $\clm_{M_{\Phi}} \neq \{0\}$. Now suppose, by Butz's characterization we have $\clm_{T_{\Phi}} = H_{\clf}^2(\D)$ for some $\clf \subseteq \cle$. From the structure of $\clm_{T_{\Phi}}$, for any $h \in \clm_{T_{\Phi}}$, we have
\[
h \in \cap_{n=1}^{\infty} \ker (I_{H_{\cle}^2(\D)} - M_{\Phi}^n M_{\Phi}^{*n}).
\]
Now note that for any $n \in \Nat$,
\[
I  - M_{\Phi}^n M_{\Phi}^{*n} = I - (M_z^n P + P^{\perp})(M_z^{*n} P + P^{\perp}) =  I - M_z^n M_z^{*n} P - P^{\perp} = (I - M_z^n M_z^{*n}) P
\]
Therefore, 
\[
h \in \cap_{n=1}^{\infty} \ker (I_{H_{\cle}^2(\D)} - M_{\Phi}^n M_{\Phi}^{*n}) \Longleftrightarrow h \in \cap_{n=1}^{\infty} \ker (I_{H_{\cle}^2(\D)} - M_z^n M_z^{*n}) P \subseteq \cle.
\]
This immediately contradicts $\clm_{T_{\Phi}} = H_{\clf}^2(\D)$ for some closed subspace $\clf \subseteq \cle$.
\end{example}
Now let us identify the situation when the decomposition considered by Butz holds.
\begin{thm}\label{Butzcorrect}
Let $T_{\Phi}$ be a contractive Toeplitz operator on $H_{\cle}^2(\D)$ with a unitary part $\clm_{T_{\Phi}}$. Then the following are equivalent:
\begin{enumerate}
\item[(i)] $\clm_{T_{\Phi}} = H_{\cle_0}^2(\D)$ for some closed subspace $\cle_0 \subseteq \cle$,
\item[(ii)] there exists a closed subspace $\cle_0 \subseteq \cle$ such that $\clm_{T_{\Phi}} \subseteq H_{\cle_0}^2(\D)$, and $T_{\Phi}|_{\clm_{T_{\Phi}}} = T_W|_{H_{\cle_0}^2(\D)}$, for some unitary $W: \cle_0 \raro \cle_0$,
\item[(iii)] there exists a closed subspace $\cle_0 \subseteq \cle$ such that $\clm_{T_{\Phi}} \subseteq H_{\cle_0}^2(\D)$, and $\cle_0$ is $\Phi(e^{it})$-reducing and $\Phi(e^{it})|_{\cle_0}$ is a constant unitary a.e. on $\T$.
\end{enumerate}
\end{thm}
\begin{proof}
Let us look in the direction $(i) \implies (ii)$. In Theorem \ref{mainthm}, we have shown that $\clm_{T_{\Phi}}$ is a $M_z$-invariant closed subspace of $H_{\cle}^2(\D)$, and thus, by the Beurling-Lax-Halmos's theorem there exists a Hilbert space $\clf$ and an inner function $\Theta(z) \in H_{\clb(\clf, \cle)}^{\infty}(\D)$ such that 
\[
\clm_{T_{\Phi}} = M_{\Theta} H_{\clf}^2(\D).
\]
Now,  if there exists a closed subspace $\cle_0 \subseteq \cle$ such that $\clm_{T_{\Phi}}  = H_{\cle_0}^2(\D)$, then
\[
M_{\Theta} H_{\clf}^2(\D) = H_{\cle_0}^2(\D) = i_{\cle_0} H_{\cle_0}^2(\D),
\]
where $i_{\cle_0} : \cle_0 \hookrightarrow \cle$ is the inclusion map.  Now, from the uniqueness of the Beurling-Lax-Halmos theorem there exists a unitary $A: \clf \raro \cle_0$ such that 
\[
\Theta(z) = i_{\cle_0} A.
\]
From condition (\ref{maincondn}) we know that $\clm_{T_{\Phi}} \neq \{0\}$, if and only if there exists a unitary $U: \clf \raro \clf$ such that
\[
\Phi(e^{it}) \Theta(e^{it}) =  \Theta(e^{it}) U, \quad (\text{ a.e. on } \T).
\]
and therefore, we get that $\Phi(e^{it})|_{\cle_0} = i_{\cle_0} A U A^* i_{\cle_0}^*: \cle_0 \subseteq \cle \raro \cle_0 \subseteq \cle $, a.e. on $\T$. Similarly, the other identity
\[
\Phi(e^{it})^* \Theta(e^{it}) =  \Theta(e^{it}) U^*, \quad (\text{ a.e. on } \T).
\]
gives $\Phi(e^{it})^*|_{\cle_0} = i_{\cle_0} A U^* A^* i_{\cle_0}^*$, a.e. on $\T$. This shows that $\Phi(e^{it})|_{\cle_0}$ is a constant unitary a.e. on $\T$ and moreover,  the above identities imply that for any $\eta \in \cle_0$ and $k \in \Nat$ we get
\[
T_{\Phi}z^k \eta  = L_{\Phi} z^k \eta = L_{\Phi}M_z^k  \eta = M_z^k L_{\Phi} \eta = M_z^k L_{\Phi|_{\cle_0}} \eta = z^k i_{\cle_0} A U A^* i_{\cle_0}^* \eta,
\]
which further implies that 
\[
z^k \eta \in \ker (T_{\Phi}|_{H_{\cle_0}^2(\D)} -  T_W|_{H_{\cle_0}^2(\D)}),
\]
where $W = i_{\cle_0} A U A^* i_{\cle_0}^*$. Since $H_{\cle_0}^2(\D) = \overline{\mbox{span }}\{z^k \eta: k \in \Nat, \eta \in \cle_0\}$, we get
\[
T_{\Phi}|_{H_{\cle_0}^2(\D)}  = T_W|_{H_{\cle_0}^2(\D)}.
\]
Similarly, we will get 
\[
T_{\Phi}^*|_{H_{\cle_0}^2(\D)}  =  T_W^*|_{H_{\cle_0}^2(\D)}.
\]
In other words,  we prove that $T_{\Phi}|_{\clm_{T_{\Phi}}}$ is a constant unitary $T_W$ on $\clm_{T_{\Phi}} = H_{\cle_0}^2(\D)$.
For proving the direction $(ii) \implies (iii)$,  let us first note that $T_{\Phi}|_{\clm_{T_{\Phi}}} = T_W |_{H_{\cle_0}^2(\D)}$, implies that
\[
H_{\cle_0}^2(\D) = T_{W}|_{H_{\cle_0}^2(\D)} H_{\cle_0}^2(\D) = T_{\Phi}|_{\clm_{T_{\Phi}}} H_{\cle_0}^2(\D)  = T_{\Phi} \clm_{T_{\Phi}} =  \clm_{T_{\Phi}}.
\]
The last but one equality follows since $\clm_{T_{\Phi}} \subseteq H_{\cle_0}^2(\D)$.  Thus, we have 
\[
T_{\Phi} |_{H_{\cle_0}^2(\D)}  = T_{W} |_{H_{\cle_0}^2(\D)},
\]
and therefore, 
\[
T_{{(\Phi |_{\cle_0}} - W)} |_{H_{\cle_0}^2(\D)} = 0.
\]
Since, $H_{\cle_0}^2(\D)$ is both $T_{\Phi}$ and $T_{W}$ reducing subspace, we can think $T_{(\Phi |_{\cle_0} - W)}$ as a Toeplitz operator on $H_{\cle_0}^2(\D)$. This implies that
\[
\Phi(e^{it})|_{\cle_0} = W \quad (\text{ a.e. on } \T).
\]
Similarly, we will get
\[
\Phi(e^{it})^*|_{\cle_0} = W^* \quad (\text{ a.e. on } \T).
\]
For the direction, $(iii) \implies (i)$, let $\Phi(e^{it})|_{\cle_0} = W$ (a constant unitary) a.e. on $\T$. Thus for any $\eta \in \cle_0$ and $k \in \Nat$,
\[
T_{\Phi} \eta = P_{H_{\cle}^2(\D)}  L_{\Phi}  \eta = P_{H_{\cle}^2(\D)}   L_{\Phi|_{\cle_0}} \eta = P_{H_{\cle}^2(\D)}   W \eta =  W \eta,
\]
and 
\[
T_{\Phi}^*  \eta = P_{H_{\cle}^2(\D)} L_{\Phi}^*  \eta = P_{H_{\cle}^2(\D)}  L_{\Phi|_{\cle_0}}^* \eta = P_{H_{\cle}^2(\D)} W^* \eta =  W^* \eta,
\]
shows that
\[
\cle_0 \subseteq \clm_{T_{\Phi}}.
\]
Since $\clm_{T_{\Phi}}$ is a $M_z$-invariant closed subspace of $H_{\cle}^2(\D)$, we have $H_{\cle_0}^2(\D) \subseteq \clm_{T_{\Phi}}$. From our assumption we already have $\clm_{T_{\Phi}} \subseteq H_{\cle_0}^2(\D)$, which implies that $\clm_{T_{\Phi}} = H_{\cle_0}^2(\D)$. This completes the proof.
\end{proof}

\begin{rem}
We point out that the gap in the proof of the main theorem in \cite{Butz} lies in Line 17 of page 97 in the following sentence
\[
`` \ldots \text{ Also for } S^*x = \sum_{n=0}^{\infty} S^n c_{n+1} ".
\]
Here, the author had considered $x= \sum_{n=0}^{\infty} S^n c_n$, where $c_n \in C_0$ for all $n \in \Nat$.  The space $C_0 := \clm \ominus S\clm$, where $\clm$ is the unitary part of the generalized Toeplitz operator $T$ on $\clh$. For the above expression (within quotes) to hold one should have $c_0 \in C \cap \clm$. However, the author did not prove this fact. The above expression is then used to prove that $\clm$ is $M_z$-reducing, which is again used to prove that $C_0 = C \cap \clm $ in Line 21.  In a certain sense,  the author's argument is circular and incorrect, as can be seen in the example and main result of this section.
\end{rem}

\section{Analytic Toeplitz operators}\label{analytic-toeplitz}
We shall now focus on analytic Toeplitz operators on $H_{\cle}^2(\D)$. Let us briefly recall that any $\Phi(z) \in H_{\clb(\cle)}^{\infty}(\D)$ is a Schur function that is, $\|\Phi\|_{\infty} \leq 1$, if and only if there exists a transfer-function realization for $\Phi(z)$ (for instance, see \cite[Theorem 1.1]{BBF}). More precisely, there exists an auxiliary Hilbert space $\clk$ and a unitary
\[
W := \begin{bmatrix}
A & B \\
C & D
\end{bmatrix}: \cle \oplus \clk \raro  \cle \oplus \clk
\]
such that 
\begin{equation}\label{transfer}
\Phi(z) := \tau_W(z) = A + z B (I_{\cle} - z D)^{-1} C.
\end{equation}

Since $W$ is unitary we get the following properties:
\begin{enumerate}
\item[(a)] $A^*A + C^*C = I_{\cle} = B^*B + D^*D$,
\item[(b)]$A^*B  + C^*D= 0 = AC^* +BD^*$.
\item[(c)] $D D^* + CC ^*= I_{\cle} = AA^* + BB^*$,
\end{enumerate}

Using this we get the following identities for all $\lambda \in \D$,
\begin{equation}\label{transfer1}
I_{\cle}  - \Phi(\lambda) \Phi(\lambda)^* = (1 - |\lambda|^2) B (I - \lambda D)^{-1} (I - \bar{\lambda} D^*)^{-1}  B^*,
\end{equation}
and
\begin{equation}\label{transfer2}
I_{\cle}  - \Phi(\lambda)^* \Phi(\lambda) = (1 - |\lambda|^2) C^* (I - \bar{\lambda} D^*)^{-1}  (I - \lambda D)^{-1}C.
\end{equation}
The above identities also show that if we begin with a transfer function $\tau_W$, then it is contractive for all $\lambda \in \D$. The transfer-function realization is a fundamental tool used for understanding the properties of analytic Toeplitz operators. Using the identities (\ref{transfer1}) and (\ref{transfer2}),  the following observations equivalent to \cite[Theorem 3.1]{DSau} can be made.
\begin{propn}[Das and Sau]\label{DS}
Let $\Phi(z) \in H_{\clb(\cle)}^{\infty}(\D)$ such that $\|\Phi\|_{\infty} \leq 1$ and $\Phi(0)$ is not a unitary,  then the following statements
\begin{enumerate}
\item[(i)] $\Phi(0) \in \clb(\cle)$, has a  unitary part,
\item[(ii)] $\Phi(\lambda) \in \clb(\cle)$, has a  unitary part for all $\lambda \in \D$,
\item[(iii)] $M_{\Phi} \in \clb(H_{\cle}^2(\D))$, has a unitary part.
\end{enumerate}
have the following implications: $(i) \Longleftrightarrow (ii)$, and $(i) \implies (iii)$.
\end{propn}

Let us note down a result that can deduced from the above characterization. We believe this observation is new and independently interesting.
\begin{propn}
Let $\Phi(z) \in H_{\clb(\cle)}^{\infty}(\D)$ such that $\|\Phi\|_{\infty} \leq 1$. Furthermore, assume that $\Phi(0)$ is an isometry on $\cle$. Then either $M_{\Phi}$ has a unitary part or $M_{\Phi}$ is a $C_{\cdot 0}$-contraction, that is, $\{M_{\Phi}^{*n }\}_n \raro 0$ in the strong operator topology.
\end{propn}
\begin{proof}
Since $\Phi(0)$ is a isometry on $\cle$ by the Wold-von Neumann decomposition $\Phi(0)$ has  a unitary part or $\Phi(0)$ is a $C_{\cdot 0}$-isometry (also known as a \textit{shift} operator). From the above Proposition \ref{DS}, if $\Phi(0)$ has a unitary part then $M_{\Phi}$ has a unitary part.  The proof gets complete by using \cite[Theorem 3.4]{Sarkar} which shows that $M_{\Phi}$ on $H_{\cle}^2(\D)$ is a $C_{\cdot 0}$-contraction if and only $\Phi(0)$ is a $C_{\cdot 0}$-contraction.
\end{proof}
From Proposition \ref{DS}, it is clear that the missing piece and the most important direction is whether $(iii) \implies (i)$.  There is an example in \cite[Remark 3.2]{DSau}, where the authors have shown this does not hold in general.  It is evident from this discussion that we need to first analyze when $\Phi(0) \in \clb(\cle)$ has a unitary part. This is the content of the following result. 

\begin{propn}\label{analyticmainprop}
Let $\Phi(z) \in H_{\clb(\cle)}^{\infty}(\D)$ such that $\|\Phi\|_{\infty} \leq 1$, then the following are equivalent
\begin{enumerate}
\item[(i)] the operator $\Phi(0)$ on $\cle$ has a unitary part,
\item[(ii)] the constant Toeplitz operator $T_{\Phi(0)}$ on $H_{\cle}^2(\D)$ has a unitary part,
\item[(iii)] there exists a Hilbert space $\clf$, an inner function $\Theta(z) \in H_{\clb(\clf,\cle)}^{\infty}(\D)$, and a unitary $U: \clf \raro \clf$ such that
\begin{equation}\label{point1}
\Phi(0) \Theta(\lambda) = \Theta(\lambda) U; \quad \Phi(0)^* \Theta(\lambda) = \Theta(\lambda) U^* \quad (\forall \lambda \in \D).
\end{equation}
\item[(iv)] there exists a Hilbert space $\clf$, an inner function $\Theta(z) \in H_{\clb(\clf,\cle)}^{\infty}(\D)$, and a unitary $U: \clf \raro \clf$ such that
\begin{equation}\label{point2}
\Phi(\lambda) \Theta(\lambda) = \Theta(\lambda) U; \quad \Phi(\lambda)^* \Theta(\lambda) = \Theta(\lambda) U^* \quad (\forall \lambda \in \D).
\end{equation}
\end{enumerate}
\end{propn}
\begin{proof}
Our approach is to show the following directions: $(i) \implies (ii) \implies (iii) \implies (i)$ and $
(iii) \Longleftrightarrow (iv)$. It is easy to see that when $\Phi(0)$ has a unitary part in $\cle$, then the operator $T_{\Phi(0)}$ on $H_{\cle}^2(\D)$ has a unitary part. In particular, $\clm_{\Phi(0)} \subseteq \clm_{T_{\Phi(0)}}$. This implies the direction $(i) \implies (ii)$. Now let us look at $(ii) \implies (iii)$.  Let $\clm_{T_{\Phi(0)}}$ denote the unitary part of the operator $T_{\Phi(0)}$ on $H_{\cle}^2(\D)$. Note that
\[
\clm_{T_{\Phi(0)}} = \{f \in H_{\cle}^2(\D): \|\Phi(0)^n f \| = \|f\| = \|\Phi(0)^{*n} f\|, \forall n \in \Nat\}
\]
It is clear from the above description that $\clm_{T_{\Phi(0)}}$ is a $M_z$-invariant closed subspace of $H_{\cle}^2(\D)$. From assumption, $\clm_{T_{\Phi(0)}} \neq \{0\}$, and therefore, by the Beurling-Lax-Halmos theorem, there exists a Hilbert space $\clf$ and an inner function $\Theta(z) \in H_{\clb(\clf,\cle)}^{\infty}(\D)$ such that
\[
\clm_{T_{\Phi(0)}} = M_{\Theta} H_{\clf}^2(\D).
\]
Since $\clm_{T_{\Phi(0)}}$ is $T_{\Phi(0)}$-reducing, we get, from Lemma \ref{isom},
\[
T_{\Phi(0)}  M_{\Theta} M_{\Theta}^*= M_{\Theta}M_{\Theta}^* T_{\Phi(0)}.
\]
which implies that
\[
 T_{\Phi(0)}  M_{\Theta} = M_{\Theta}M_{\Theta}^* T_{\Phi(0)} M_{\Theta}.
\]
Since $T_{\Phi(0)}$ is unitary on $M_{\Theta} H_{\clf}^2(\D)$, the operator $M_{\Theta}^* T_{\Phi(0)} M_{\Theta}$ is an isometry.  In the same manner, 
\[
 T_{\Phi(0)}^* M_{\Theta}= M_{\Theta}M_{\Theta}^*  T_{\Phi(0)}^*  M_{\Theta}.
\]
will imply that $M_{\Theta}^* T_{\Phi(0)}^* M_{\Theta}$ is an isometry. From Corollary \ref{unitoep}, there must exist an unitary (say) $U:\clf \raro \clf$ such that 
\[
M_{\Theta}^* T_{\Phi(0)} M_{\Theta} = T_U.
\]
and thus, 
\[
M_{\Theta} M_{\Theta}^* T_{\Phi(0)} M_{\Theta} = M_{\Theta} T_U.
\]
Since $M_{\Theta} H_{\clf}^2(\D)$ is $T_{\Phi(0)}$-reducing, we get
\[
T_{\Phi(0)} M_{\Theta} = M_{\Theta} T_U,
\]
which further implies that 
\[
\Phi(0) \Theta(\lambda) = \Theta(\lambda) U \quad (\lambda \in \D).
\]
Similarly, from 
\[
M_{\Theta}^* T_{\Phi(0)}^* M_{\Theta} = T_U^*.
\]
we have
\[
\Phi(0)^* \Theta(\lambda) =  \Theta(\lambda) U^* \quad (\lambda \in \D).
\]
This completes the proof for the direction $(ii) \implies (iii)$. For the direction $(iii) \implies (i)$, this is straightforward, since $\Theta(z) \not\equiv 0$, there exists $\lambda \in \D$ for which $\Theta(\lambda) \neq 0$. For this $\lambda$, the condition 
\[
\Phi(0) \Theta(\lambda) = \Theta(\lambda)U; \quad \Phi(0)^* \Theta(\lambda) = \Theta(\lambda)U^*.
\]
implies that for all $n \in \Nat$,
\[
\Phi(0)^{n} \Phi(0)^{*n} \Theta(\lambda)  = \Theta(\lambda); \quad \Phi(0)^{*n} \Phi(0)^{n} \Theta(\lambda) =  \Theta(\lambda).
\]
This implies that $\Phi(0)$ has a unitary part containing the range of  $\Theta(\lambda)$. Now let us prove  $(iii) \implies (iv)$. Note that conditions (\ref{point1}) imply that
\[
M_{\Theta}^* T_{\Phi(0)}^* T_{\Phi(0)}M_{\Theta} = M_{\Theta}^*M_{\Theta} = I_{H_{\clf}^2(\D)}.
\]
Similarly, we get
\[
M_{\Theta}^* T_{\Phi(0)} T_{\Phi(0)}^*M_{\Theta} = M_{\Theta}^*M_{\Theta} = I_{H_{\clf}^2(\D)}.
\]
Now, note that 
\[
M_{\Theta}^* T_{\Phi(0)}^* T_{\Phi(0)}M_{\Theta} = M_{\Theta}^*  T_A^* T_A M_{\Theta}  = M_{\Theta}^* M_{\Theta} - M_{\Theta}^*  T_C^* T_C M_{\Theta},
\]
The last equality follows from property $(a)$ corresponding to the unitary $W$. Thus, we get 
\[
M_{\Theta}^*  T_C^* T_C M_{\Theta} = 0,
\]
which further implies that $T_C M_{\Theta} = 0$, and therefore, $C \Theta(\lambda) = 0$ for all $\lambda \in \D$.  Using the transfer-function formula (\ref{transfer}) we get
\[
\Phi(\lambda) \Theta(\lambda) = ( A + \lambda B (I_{\cle} - \lambda D)^{-1} C)\Theta(\lambda) = A \Theta(\lambda) = \Phi(0) \Theta(\lambda) = \Theta(\lambda) U.
\]
Similarly, note from condition $(c)$ that 
\[
M_{\Theta}^* T_{\Phi(0)} T_{\Phi(0)}^*M_{\Theta} = M_{\Theta}^*  T_A T_A^* M_{\Theta}  = M_{\Theta}^* M_{\Theta} - M_{\Theta}^*  T_B T_B^* M_{\Theta},
\]
Thus, we get 
\[
M_{\Theta}^*  T_B T_B^* M_{\Theta} = 0,
\]
which further implies that $ T_B^* M_{\Theta} = 0$, and therefore, $B^* \Theta(\lambda) = 0$ for all $\lambda \in \D$. This implies that 
\[
\Phi(\lambda)^* \Theta(\lambda) = ( A^* + \bar{\lambda} C^* (I_{\cle} - \bar{\lambda} D^*)^{-1} B^*)\Theta(\lambda) = A^* \Theta(\lambda) = \Phi(0)^* \Theta(\lambda) = \Theta(\lambda) U^*.
\]
This completes the proof of $(iii) \implies (iv)$. Now, if 
\[
\Phi(\lambda) \Theta(\lambda) = \Theta(\lambda) U; \quad \Phi(\lambda)^* \Theta(\lambda) = \Theta(\lambda) U^*
\]
is satisfied for all $\lambda \in \D$, then
\[
(I_{\cle} - \Phi(\lambda) \Phi(\lambda)^*) \Theta(\lambda) = \Theta(\lambda) -\Phi(\lambda) \Phi(\lambda)^*  \Theta(\lambda)
=  \Theta(\lambda) - \Theta(\lambda) U U^*
=0.
\]
and therefore,
\[
0 = \Theta(\lambda)^* (I_{\cle} - \Phi(\lambda) \Phi(\lambda)^*) \Theta(\lambda) = (1 - |\lambda|^2)  \Theta(\lambda)^* B (I - \lambda D)^{-1}  (I - \bar{\lambda}  D^*)^{-1}  B^* \Theta(\lambda).
\]
This implies that $B^* \Theta(\lambda) = 0$ for all $\lambda \in \D$. Similarly, the above conditions imply that
\[
(I_{\cle} - \Phi(\lambda^*) \Phi(\lambda)) \Theta(\lambda) = \Theta(\lambda) -\Phi(\lambda)^* \Phi(\lambda) \Theta(\lambda) 
=  \Theta(\lambda) - \Theta(\lambda) U^* U^*
=0,
\]
and therefore,
\[
0 = \Theta(\lambda)^* (I_{\cle} - \Phi(\lambda)^* \Phi(\lambda)^*) \Theta(\lambda) = (1 - |\lambda|^2)  \Theta(\lambda)^* C^* (I - \bar{\lambda} D^*)^{-1}  (I - \lambda  D)^{-1}  C \Theta(\lambda).
\]
This implies that $C \Theta(\lambda) = 0$ for all $\lambda \in \D$. Thus, we get
\[
\Phi(0) \Theta(\lambda) = \Phi(\lambda) \Theta(\lambda) = \Theta(\lambda) U 
\]
and
\[
\Phi(0)^* \Theta(\lambda) = \Phi(\lambda)^* \Theta(\lambda) = \Theta(\lambda) U^*. 
\]
This completes the proof that $(iv) \implies (iii)$, and also the proof of this proposition.
\end{proof}
Now we are ready to establish the conditions that force $\Phi(0)$ to have a unitary part. 
\begin{thm}\label{analyticmainthm}
Let $\Phi \in H_{\clb(\cle)}^{\infty}(\D)$ with $\|\Phi\|_{\infty} \leq 1$.  If $M_{\Phi}$ has a unitary part $\clm_{M_{\Phi}}$, then the condition $\| T_{\Phi(0)} f \| = \|f\|$ for all $f \in \clm_{M_{\Phi}}$ implies that $\Phi(0)$ has a  unitary part.
\end{thm}
\begin{proof}
Since $M_{\Phi}$ has a unitary part, from Theorem \ref{mainthm}, we know that there must exist a Hilbert space $\clf$ and a unitary $U:\clf \raro \clf$ such that
\[
\Phi(e^{it}) \Theta(e^{it}) = \Theta(e^{it}) U; \quad \Phi(e^{it})^* \Theta(e^{it}) = \Theta(e^{it}) U^* \quad (\text{ a.e.  on } \T).
\]
Now $\Phi(z)$ is a bounded analytic function further implies that 
\begin{equation}\label{p1}
\Phi(\lambda) \Theta(\lambda) = \Theta(\lambda) U \quad (\forall \lambda \in \D).
\end{equation}
From Proposition \ref{analyticmainprop}, it is clear that we have to find conditions that gives 
\begin{equation}\label{p2}
\Phi(\lambda)^* \Theta(\lambda) = \Theta(\lambda) U^* \quad (\forall \lambda \in \D).
\end{equation}
We claim that condition (\ref{p1}) will imply condition (\ref{p2}) if $\|T_{\Phi(0)} f\| = \|f\|$ for all $f \in \clm_{M_{\Phi}}$.  To prove this claim, let us first fix a $\lambda \in \D$ and note that from condition (\ref{p1}), we get
\[
\Phi(\lambda)^* \Phi(\lambda) \Theta(\lambda)U^* = \Phi(\lambda)^* \Theta(\lambda)
\]
and thus, 
\[
- \big(I - \Phi(\lambda)^* \Phi(\lambda) \big) \Theta(\lambda)U^*  = \Phi(\lambda)^* \Theta(\lambda) - \Theta(\lambda)U^*.
\]
This is equivalent to
\[
- (1 - |\lambda|^2) C^* (I - \bar{\lambda} D^*)^{-1}  (I - \lambda D)^{-1}C\Theta(\lambda)U^* = \Phi(\lambda)^* \Theta(\lambda) - \Theta(\lambda)U^*.
\]
Hence, $ \Phi(\lambda)^* \Theta(\lambda) = \Theta(\lambda)U^*$ if and only if 
\[
C^* (I - \bar{\lambda} D^*)^{-1}  (I - \lambda D)^{-1}C\Theta(\lambda)=0.
\]
This is again equivalent to 
\[
\Theta(\lambda)^* C^* (I - \bar{\lambda} D^*)^{-1}  (I - \lambda D)^{-1}C\Theta(\lambda)=0,
\]
and therefore, equivalent to the condition 
\[
(I - \lambda D)^{-1}C\Theta(\lambda)=0.
\]
Now acting on the left by $(I - \lambda D)$, we get, $C\Theta(\lambda)=0$. Thus, 
\[
\Phi(\lambda)^* \Theta(\lambda) = \Theta(\lambda)U^* \text{ if and only if }C \Theta(\lambda) = 0,
\] 
Using the fact that $W$ is a unitary, we know $A^*A + C^* C = I_{\cle}$. Thus, $C \Theta(\lambda) = 0$ is equivalent to
\[
A^*A \Theta(\lambda) =  \Theta(\lambda) \quad (\lambda \in \D),
\]
Since $\lambda \in \D$ was arbitrary, we observe that condition (\ref{p2}) holds if and only if $A^*A \Theta(\lambda) =  \Theta(\lambda)$ for all $\lambda \in \D$. The later condition is equivalent to $T_A^*T_A M_{\Theta} = M_{\Theta}$, which is again equivalent to 
\[
M_{\Theta}^* T_A^* T_A M_{\Theta} =  I_{H_{\clf}^2(\D)}
\]
or in other words, 
\[
\|T_{\Phi(0)} f\| = \|T_{A} f\| = \|f\|.
\]
for all $f \in \clm_{T_{\Phi}}$.  This is because, $M_{\Theta}^* T_A^* T_A M_{\Theta} =  I_{H_{\clf}^2(\D)}$ will imply that $M_{\Theta}^* T_C^* T_C M_{\Theta} = 0$, that is, $T_C M_{\Theta} = 0$, which is again equivalent to $T_A^*T_A M_{\Theta} =  M_{\Theta}$. This completes the proof.
\end{proof}
It is evident from the above result that we need to consider constant operator $T_{\Phi(0)}$ on the full Hardy space $H_{\cle}^2(\D)$ for finding the correspondence between the unitary part of $\Phi(0)$ and the unitary part of $T_{\Phi}$ on $H_{\cle}^2(\D)$.  We will now focus on a special class of Schur functions. In particular, we consider transfer functions corresponding to the following unitary,
\[
W:= \begin{bmatrix}
A & B\\
C & 0
\end{bmatrix}: \cle \oplus \clk \raro \cle \oplus \clk,
\]
that is, $\Phi(z):= \tau_W(z) = A + zBC$.  Here, we have the following properties.
\begin{enumerate}
\item[(a)] $A^*A + C^*C = I_{\cle} = AA^* + BB^*$,
\item[(b)] $A^*B = 0 = AC^*$,
\item[(c)] $B$ and $C^*$ are isometries, and $A$ is a partial isometry on $\cle$.
\end{enumerate}
Although the above class of functions seems restrictive, it serves an important role in finding analytic models for abstract operators on Hilbert spaces. For instance, Berger, Coburn, and Lebow in \cite{BCL} proved that tuple of commuting isometries whose product is a shift is unitarily equivalent to tuple of Toeplitz operators on certain $H_{\cle}^2(\D)$, corresponding to symbols of the form $U(zP+P^{\perp})$, where $P$ is an orthogonal projection and $U$ is a unitary on $\cle$. We have the following result for Toeplitz operators corresponding to these symbols.
\begin{thm}
Let $\Phi(z) \in H_{\clb(\cle)}^{\infty}(\D)$ such that $\Phi(z): = \tau_W(z) = A+zBC$, such that $M_{\Phi}$ has a unitary part, then any one of the following conditions
\begin{enumerate}
\item[(i)] $\|T_{\Phi(0)}f\| = \|f\|$ for all $f \in \clm_{M_{\Phi}}$,
\item[(ii)] $\|T_{\Phi(0)}^*f\| = \|f\|$ for all $f \in \clm_{M_{\Phi}}$, and $\ker \Phi(0) \perp \clm_{M_{\Phi}}$,
\end{enumerate}
will imply that $\Phi(0)$ has a unitary part.
\end{thm}
\begin{proof}
If $M_{\Phi}$ has a unitary part, then from Theorem \ref{mainthm}, there exist a Hilbert space $\clf$ and an inner function $\Theta(z) \in H_{\clb(\clf,\cle)}^{\infty}(\D)$ such that 
\begin{equation}\label{p6}
(A + \lambda BC) \Theta(\lambda) = \Phi(\lambda) \Theta(\lambda) = \Theta(\lambda) U \quad (\forall \lambda \in \D),
\end{equation}
where $\clm_{T_{\Phi}} = M_{\Theta} H_{\clf}^2(\D)$.  Let us fix a $\lambda \in \D$. Now, acting on the left of both sides by $B^*$ and using condition $(b)$, we get
\[
\lambda B^*B C\Theta(\lambda) = B^* \Theta(\lambda)U
\] 
and thus using condition $(c)$ it implies that
\begin{equation}\label{sp3}
\lambda C\Theta(\lambda) = B^* \Theta(\lambda)U,
\end{equation}
Since $\lambda \in \D$ was arbitrary, the above identity for $\lambda \neq 0$ gives
\begin{equation}\label{p7}
C \Theta(\lambda) = 0 \text{ if and only if }  B^* \Theta(\lambda) = 0.
\end{equation}
For $\lambda = 0$, we get $B^* \Theta(0) = 0$, and therefore, $AA^* \Theta(0) = \Theta(0)$. Now, let us again go back to condition (\ref{p6}). Acting on the left  by $A^*$ and using condition $(b)$, we get
\begin{equation}\label{p4}
A^* A \Theta(\lambda) = A^* \Theta(\lambda)U.\end{equation}
Since $A$ is a partial isometry, acting on the left of both sides by $A$, we get
\begin{equation}\label{p5}
A A^* \Theta(\lambda) = A \Theta(\lambda)U^*.
\end{equation}
From Proposition \ref{analyticmainprop}, we know that $A = \Phi(0)$ has a unitary part if and only if 
\begin{equation}\label{p3}
A \Theta(\lambda) = \Theta(\lambda) U; \quad A^*\Theta(\lambda) = \Theta(\lambda) U^* \quad (\forall \lambda \in \D)
\end{equation}  
Since $\lambda \in \D$ was arbitrary, it follows from the observations (\ref{p4}) and (\ref{p5}) that the above conditions are true if and only if 
\[
A^* A \Theta(\lambda) = \Theta(\lambda); \quad AA^* \Theta(\lambda) = \Theta(\lambda) \quad (\forall \lambda \in \D).
\]
In turn, these conditions are equivalent to $T_A T_A^* M_{\Theta} = M_{\Theta}$ and $T_A^* T_A M_{\Theta} = M_{\Theta}$.  Using properties of the unitary $W$, we see that these conditions are again equivalent to $ T_B^* M_{\Theta} = 0$ and $T_C M_{\Theta} =  0$, respectively.  From condition (\ref{sp3}), it follows that $B^* \Theta(\lambda) = 0$ if and only if $C \Theta(\lambda) = 0$ for all $\lambda \in \D \setminus \{0\}$. Since $B^* \Theta(0) = 0$, we can deduce that $T_C M_{\Theta} = 0$ will imply $T_B^* M_{\Theta} = 0$ but not vice versa. This is because, $B^* \Theta(0) = 0$ may not imply $C \Theta(0) = 0$. So we need to find an extra condition along with $T_B^* M_{\Theta} = 0$ that forces $C \Theta(0) = 0$. We claim that this condition is equivalent to $\ker \Phi(0) \perp \clm_{M_{\Phi}}$. Since $A = \Phi(0)$ is a partial isometry this is again equivalent to 
\[
(I_{\cle}  - \Phi(0)^* \Phi(0)) \cle \perp \clm_{M_{\Phi}}.
\]
Now, note that for any $\eta \in \cle$ and $f \in H_{\clf}^2(\D)$,
\[
\langle (I_{\cle} - \Phi(0)^* \Phi(0)) \eta,  M_{\Theta} f \rangle = \langle (I_{\cle} - A^* A) \eta,  M_{\Theta} f \rangle = \langle M_{\Theta}^* C^*C  \eta,   f \rangle =  \langle  \Theta(0)^* C^* C  \eta,  f(0) \rangle.
\]
Thus, $\ker \Phi(0) \perp \clm_{M_{\Phi}}$ is equivalent to $ \Theta(0)^* C^* C = 0$. Now, $C^*$ is an isometry further implies that $ \Theta(0)^* C^* C = 0$ if and only if $ \Theta(0)^* C^* = 0$, that is, $C \Theta(0) = 0$. This completes the proof of the claim. So, we have proved that any one of the following conditions 
\[
T_{C} M_{\Theta} = 0,
\]
and 
\[
T_{B}^* M_{\Theta} = 0 \text{ and } \ker \Phi(0) \perp \clm_{M_{\Phi}},
\]
will imply the existence of the unitary part of $\Phi(0)$. Note that using the properties of unitary $W$, the above conditions are again equivalent to 
\[
 M_{\Theta}^* T_A^* T_A M_{\Theta} =  I_{H_{\cle}^2(\D)},
\]
and
\[
 M_{\Theta}^* T_A T_A^* M_{\Theta} =  I_{H_{\cle}^2(\D)} \text{ and }  \ker \Phi(0) \perp \clm_{M_{\Phi}},
\]
respectively. In turn, these are again equivalent to
\[
\|T_{\Phi(0)}f\| = \|f\|,
\]
and
\[
\|T_{\Phi(0)}^*f\| = \|f\| \text{ and }  \ker \Phi(0) \perp \clm_{M_{\Phi}},
\]
respectively. This completes the proof.
\end{proof}
%

\section{Concluding remarks and questions}\label{conclusion}
An important aspect of any c.n.u. contraction is that it admits a $H^{\infty}$-calculus (see \cite[Chapter 3]{NF}). We can deduce the following result from Theorem \ref{mainthm}.
\begin{thm}\label{cnu-toep}
Let $T_{\Phi}$ be a Toeplitz operator on $H_{\cle}^2(\D)$ such that any one of the following conditions hold, 
\begin{enumerate}
\item[(I)] $\Phi(z)$ is completely non-unitary on some subset of $\T$ with positive measure,
\item[(II)] $\Phi(z)$ has a unitary part a.e. on $\T$, but for any Hilbert space $\clf$, and unitary $U \in \clb(\clf)$, there does not exist any inner function $\Theta(z) \in H_{\clb(\clf, \cle)}^{\infty}(\D)$ satisfying \[
\Phi(z) \Theta(z) = \Theta(z) U; \quad \Phi(z)^* \Theta(z) = \Theta(z) U^* \quad (\text{a.e.  on } \T).
\]
\end{enumerate}
Then $T_{\Phi}$ is completely non-unitary and therefore, for any non-zero function $u \in H^{\infty}(\D)$ with $|u(z)|< 1$ for all $z \in \D$, the operator $u(T_{\Phi})$ exists and is a completely non-unitary contraction.
\end{thm}
Let us highlight several natural questions worthy of further investigation.

\begin{question}
Characterize analytic Toeplitz operators $M_{\Phi}$ on $H_{\cle}^2(\D)$ for which, $M_{\Phi}$ is completely non-unitary if and only if $\Phi(0)$ is complete non-unitary on $\cle$.
\end{question}

\begin{question}
Characterize Toeplitz operators $T_{\Phi}$ belonging to the class of $C_{\cdot 0}$-contractions.
\end{question}

\begin{question}
Characterize Toeplitz operators $T_{\Phi}$ belonging to the class of $C_{0}$-contractions.
\end{question}

The answer to the above questions is unknown even in the scalar-valued Hardy spaces. However, if we consider analytic Toeplitz operators on both scalar/vector-valued Hardy spaces, we have a definite answer on the characterization of $C_{\cdot 0}$ contractions (see \cite{CIL, Sarkar2}).
\section*{Acknowledgement}
The second named author would like to thank Jaydeb Sarkar for informing him about \cite[Theorem 3.6, Theorem 3.7]{RR}. The Department of Science and Technology supports the second named author via the INSPIRE Faculty fellowship IFA19-MA141.

\end{document}